\documentclass{amsart}

\usepackage[utf8]{inputenc}
\usepackage{amssymb}
\usepackage{amsfonts}
\usepackage{amsmath}
\usepackage{amsthm}
\usepackage{enumerate}
\usepackage{mathrsfs}
\usepackage[pdftex]{graphicx}
\usepackage{wrapfig}
\usepackage{comment}
\usepackage{graphicx}
\usepackage{mathtools}
\usepackage{hyperref}
\usepackage{cleveref}
\usepackage{xcolor}
\hypersetup{
    colorlinks,
    linkcolor={red!50!black},
    citecolor={blue!50!black},
    urlcolor={blue!50!black}
}

\makeatletter
\newtheorem*{rep@thm}{\rep@title}
\newcommand{\newreptheorem}[2]{%
\newenvironment{rep#1}[1]{%
 \def\rep@title{#2 \ref{##1}}%
 \begin{rep@thm}}%
 {\end{rep@thm}}}
\makeatother

\newtheorem{thm}{Theorem}[section]
\newtheorem{cor}[thm]{Corollary}
\newtheorem{prop}[thm]{Proposition}

\newtheorem{lem}[thm]{Lemma}

\newreptheorem{thm}{Theorem}
\newreptheorem{lem}{Lemma}

\theoremstyle{definition}
\newtheorem{defn}[thm]{Definition}

\newtheorem{exmp}[thm]{Example}

\theoremstyle{remark}
\newtheorem*{rem}{Remark}

\newcommand{\Z}{\mathbb Z}

\newcommand{\R}{\mathbb R}

\newcommand{\la}{\langle}
\newcommand{\ra}{\rangle}

\let\epsilon\varepsilon

\let\phi\varphi
\let\bdy\partial

\newcommand{\case}[1]{\vspace{6pt}\noindent\textit{#1}}
\newcommand{\casesend}{\vspace{6pt}}

\newcommand{\calC}{\mathcal C}
\newcommand{\laa}{\la\kern-1.2ex~\la}
\newcommand{\raa}{\ra\kern-1.2ex~\ra}

\DeclareMathOperator{\Ind}{Ind}     
\DeclareMathOperator{\maxdeg}{maxdeg}       
\DeclareMathOperator{\mindeg}{mindeg}       
\DeclareMathOperator{\mir}{mir}     
\DeclareMathOperator{\rev}{rev}     
\DeclareMathOperator{\rot}{rot}     
\DeclareMathOperator{\Seq}{Seq}     
\DeclareMathOperator{\sign}{sign}   
\DeclareMathOperator{\ssgn}{ssgn}   
\DeclareMathOperator{\sym}{sym}     
\DeclareMathOperator{\wri}{wr}      

\newcommand{\supscr}{\textsuperscript}

\usepackage[backend=biber,sorting=ynt]{biblatex}
\title{Signed Heights of Knotoids}
\author{Larsen Linov}

\begin{document}

\begin{abstract}
The height of a knotoid is a measure of how far it is from being a knot. Here we define the positive and negative parts of the height, and we prove that they determine the unsigned height. Some polynomial invariants provide lower bounds for the signed heights. We also study a set of sequences associated to a knotoid.
\end{abstract}

\maketitle

\section{Introduction}

The theory of knotoids, introduced by Turaev in \cite{Tur12}, is an extension of classical knot theory. Knotoids have been studied recently in \cite{GuKa17}, \cite{BBHL18}, \cite{KIL18}, and \cite{KoTa20}, and they have also been used for studying proteins.

\subsection{Knotoids and Sign Sequences}

A \emph{knotoid diagram} is an immersion of an interval into $S^2$ with only transverse double crossings, together with over/under crossing information. A \emph{knotoid} is a class of knotoid diagrams up to planar isotopy and the Reidemeister moves performed away from the endpoints. (We will provide a geometric definition as well, in \Cref{sec:theta-curves}.) Knotoids are always considered to be oriented, that is, the endpoints are labelled as the \emph{tail} $v_0$ and the \emph{head} $v_1$.

A \emph{shortcut} for a knotoid diagram $K$ is an embedded path from $v_0$ to $v_1$ that intersects $K$ transversely and does not intersect the crossings. Of course, every knotoid diagram has many shortcuts. The intersections between $K$ and a shortcut are signed; see \Cref{fig:shortcut-int-sign}. (The endpoints are not considered to be intersections.)

\begin{figure}
\includegraphics[scale=0.25]{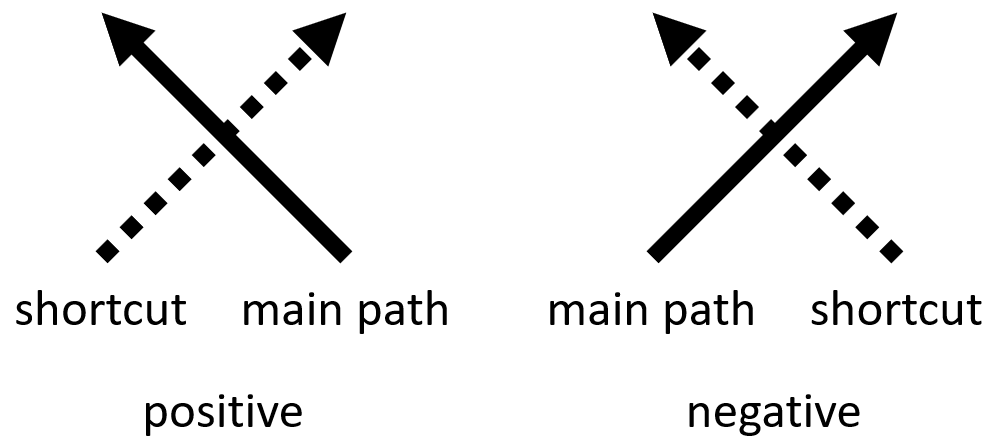}
\caption{The signs of intersections with a shortcut.}
\label{fig:shortcut-int-sign}
\end{figure}

A pair $(K, a)$, where $a$ is a shortcut for $K$, will be called a \emph{shortcut diagram}. We can connect any two shortcut diagrams for a knotoid by planar isotopy, the Reidemeister moves away from the shortcut, and the three \emph{shortcut moves} shown in \Cref{fig:shortcut-moves}.

\begin{figure}
\includegraphics[scale=0.25]{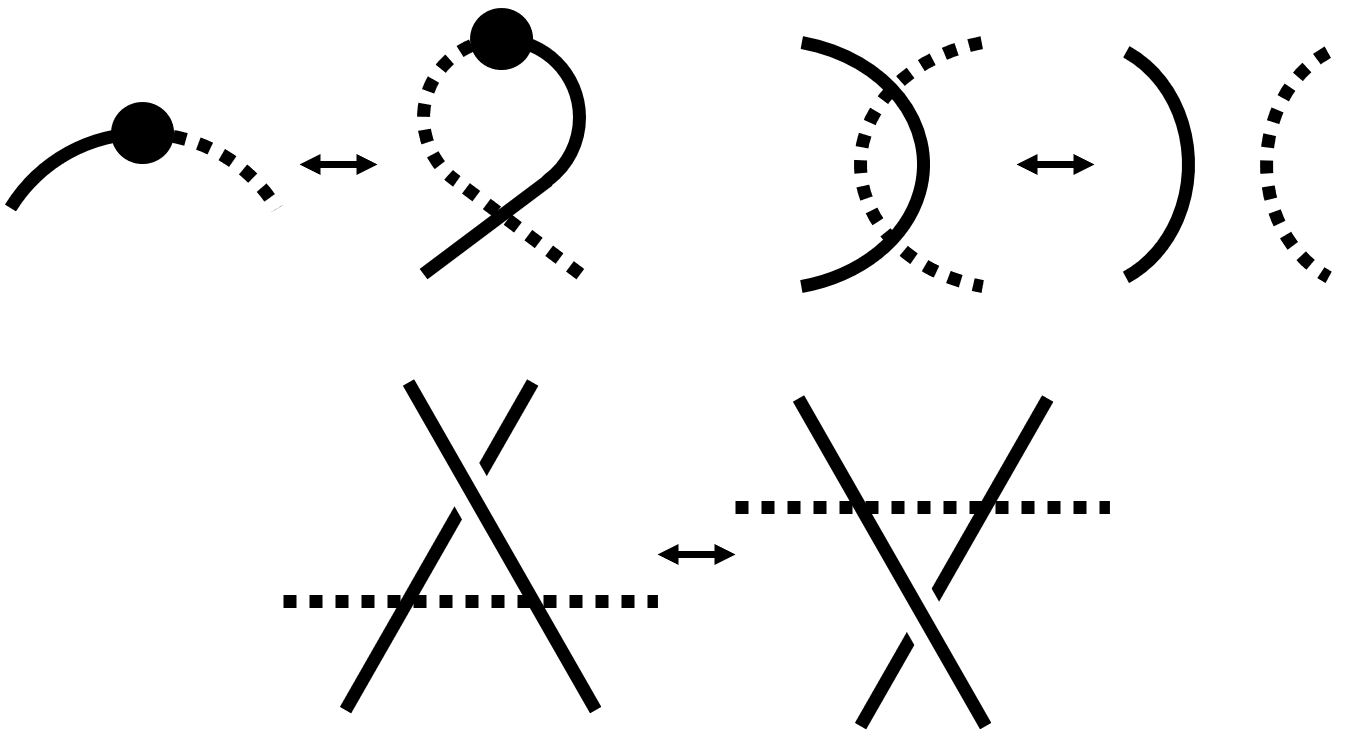}
\caption{The shortcut moves, Types I, II, and III.}
\label{fig:shortcut-moves}
\end{figure}

\begin{figure}
\includegraphics[scale=0.25]{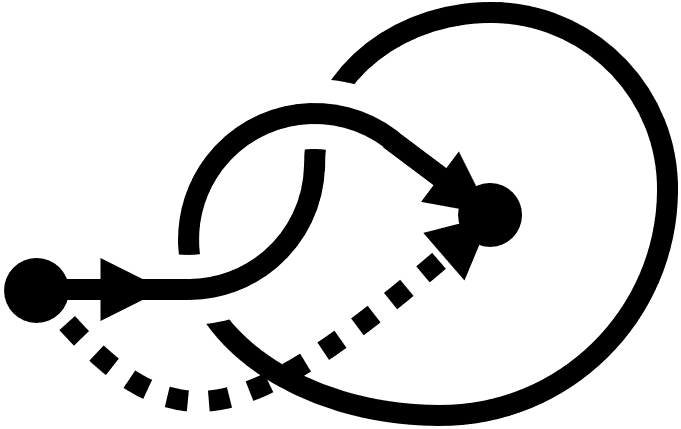}
\caption{The bifoil, an example of a nontrivial knotoid. The sign sequence of this shortcut diagram is $(+)$.}
\label{fig:bifoil}
\end{figure}

For each diagram $K$ and shortcut $a$, the \emph{height} $h(K, a)$ of the pair is the number of intersection points with $a$. A \emph{sign sequence} is a finite sequence with values in $\{+, -\}$, and for a shortcut diagram $(K, a)$ we define $\Seq(K, a)$ to be the sign sequence of length $h(K, a)$ expressing the signs of the intersections between $K$ and $a$ in the order they appear when following $K$ from $v_0$ to $v_1$. The \emph{positive} (resp. \emph{negative}) \emph{height} $h_\pm(K, a)$ is the number of appearances of $+$ (resp. $-$) in $\Seq(K, a)$.

These values give rise to natural invariants of knotoids:

\begin{defn}
For a knotoid $k$, the \emph{height} of $k$ is the minimum of the heights $h(K, a)$ over all shortcut diagrams representing $k$. (This is the \emph{complexity} in \cite{Tur12}.) We define the \emph{signed heights} $h_\pm(k)$ similarly.
\end{defn}

\begin{defn}
A sign sequence is \emph{attainable} for $k$ if it is $\Seq(K, a)$ for some shortcut diagram representing $k$. We will care in particular about the \emph{minimal} attainable sequences, that is, those realizing the height.
\end{defn}

\begin{rem}
There are several related theories not considered in this paper. In particular, knotoids on $\R^2$, virtual knotoids, and multi-knotoids are interesting generalizations.
\end{rem}

\subsection{Main Results}

There is a simple relationship between the height of a knotoid $k$ and its signed heights.

\begin{thm}\label{thm:main-thm}
For all $k$, $h(k) = h_+(k) + h_-(k)$.
\end{thm}

\Cref{thm:main-thm} reduces questions about the height of a knotoid to questions about its signed heights, which form a \emph{height pair} $(h_+, h_-)$. This will make it easier to compute the heights of some knotoids. The theorem also implies that all minimal attainable sequences for a knotoid are rearrangements of each other. The next theorem provides another restriction on the set of minimal attainable sequences.

For a sign sequence $A$, a \emph{left shift move} of size $n$ on $A$ is the result of deleting $n$ appearances of $(-, +)$ as a consecutive subsequence and then inserting $n$ copies of $(+, -)$. Similarly, a \emph{right shift move} deletes copies $(+, -)$ and adds copies of $(-, +)$. The deletions and insertions all happen at the same time. For example, a nontrivial shift move on $(-, -, +, +)$ must be a left shift of size $1$, deleting the second and third entries. The possible results after inserting $(+, -)$ are $(+, -, -, +)$, $(-, +, -, +)$, and $(-, +, +, -)$.

\begin{thm}\label{thm:shift-moves}
If $A$ and $A'$ are two minimal attainable sequences for $k$, there is a sequence of minimal attainable sequences
$$A = A_0, A_1 \ldots, A_n = A'$$
such that each $A_{i+1}$ differs from $A_i$ by a shift move.
\end{thm}

The signed heights of a knotoid can be bounded by some polynomial invariants, in particular the Turaev polynomial $\laa k\raa_\circ$ of \cite{Tur12} and the index polynomial $F_k$ of \cite{KIL18}. For a nonzero Laurent polynomial $p(t) \in \Z[t^{\pm1}]$, we will write $\deg^+(p)$ for $\max\{\maxdeg(p), 0\}$ and $\deg^-(p)$ for $\max\{-\mindeg(p), 0\}$. We also set $\deg^\pm(0) = 0$. For a Laurent polynomial in multiple variables, the signed degree in a specific variable will be denoted by (for example) $\deg_t^\pm$.

\begin{thm}\label{thm:writhe-bound}
For a knotoid $k$, $h_\pm(k) \geq \deg^\pm(F_k)$.
\end{thm}

\begin{thm}\label{thm:turaev-bound}
For a knotoid $k$, $2h_\pm(k) \geq \deg_u^\mp(\laa k\raa_\circ)$.
\end{thm}

The index polynomial also gives more specific information about attainable sign sequences.

\begin{thm}\label{thm:writhe-sequences}
Any attainable sequence for $k$ must have a consecutive subsequence adding up to $\deg^+(F_k)$, and a consecutive subsequence adding up to $-\deg^-(F_k)$.
\end{thm}

In the theorem above, of course, we treat $+$ terms as $+1$ and $-$ as $-1$.

\begin{thm}\label{thm:unique-sequences}
If $k$ is a knotoid such that the bounds in \Cref{thm:writhe-bound} are equalities, then $k$ has a unique minimal attainable sign sequence.
\end{thm}

\subsection{Organization}

In \Cref{sec:background}, we give background information on knotoids, including basic operations and the relationship between knotoids and theta-curves. In \Cref{sec:moving-disks}, we prove \Cref{thm:main-thm} and \Cref{thm:shift-moves}. \Cref{sec:properties} addresses the behavior of height and sign sequences under knotoid operations. In \Cref{sec:poly-bounds}, we recall background on $n$-writhes and the Turaev polynomial, and we prove Theorems \ref{thm:writhe-bound} through \ref{thm:unique-sequences}. \Cref{sec:applications} contains interesting examples and applications to knotoids of low height.

\subsection{Acknowledgements}

I would like to thank my advisor, Ian Agol, for his thorough and careful feedback. I am also grateful to Kyle Miller for our very helpful conversations.

\section{Background on Knotoids and Theta-Curves}\label{sec:background}

\subsection{Closures and Knot-Type Knotoids}

The first examples of knotoid invariants are the \emph{over-} and \emph{underpass closures}. A shortcut diagram for $k$ gives rise to an oriented knot diagram by incorporating the shortcut into the diagram to create an immersion of a circle into $S^1$. By taking the shortcut to pass over (resp. under) the knotoid at each crossing, we obtain a diagram of the overpass (resp. underpass) closure of $k$, denoted $k_+$ (resp. $k_-$).

Conversely, given a diagram of an oriented knot $\kappa$, and a point on an edge of the diagram, we may obtain a knotoid diagram by deleting an open interval around the chosen point. The resulting knotoid depends only on $\kappa$ and is denoted $\kappa^\bullet$. By construction, $\kappa^\bullet$ has height $0$ and $(\kappa^\bullet)_\pm = \kappa$. Any knotoid of height $0$ may be obtained in this way, and so we have a natural identification of the set of height-$0$ knotoids with oriented knots. Such knotoids are called \emph{knot-type}. Knotoids that are not knot-type are \emph{proper}.

\subsection{Multiplication}

The set of knotoid types has a natural noncommutative product. Given knotoid diagrams $K_1$ and $K_2$, we may form a diagram $K_1K_2$ by deleting small open disks around the head of $K_1$ and tail of $K_2$, then gluing appropriately along the boundaries. The resulting knotoid depends only on the knotoids represented by $K_1$ and $K_2$. It also makes sense to refer to the product of two shortcut diagrams as another shortcut diagram.

Knotoid multiplication is associative. The \emph{trivial} knotoid, the one that can be drawn without crossings, is an identity. The over/underpass closure operations and the $\kappa \mapsto \kappa^\bullet$ operation are monoid homomorphisms.

A \emph{prime} knotoid is one that cannot be written as a nontrivial product. Every knotoid has a unique decomposition of the form
$$\kappa^\bullet k_1k_2\cdots k_n,$$
where each $k_i$ is a proper prime knotoid. A knot-type knotoid is prime if and only if the corresponding knot is prime. Two distinct prime knotoids commute if and only if one or both is knot-type (\cite{Tur12}).

\subsection{Basic Involutions}

For a knotoid $k$, the \emph{reverse} $\rev(k)$ is obtained by switching the orientation on a diagram of $K$, that is, swapping the labels of the vertices. The \emph{mirror image} $\mir(k)$ is obtained from switching the over/under information on each crossing, and the \emph{symmetry} operation acts by reversing the orientation of the ambient $S^2$. \emph{Rotation} is the composition of symmetry and mirror image reflection. See \Cref{fig:involutions}. The basic involutions generate a group isomorphic to $(\Z/2\Z)^3$.

\begin{figure}
\includegraphics[scale=0.25]{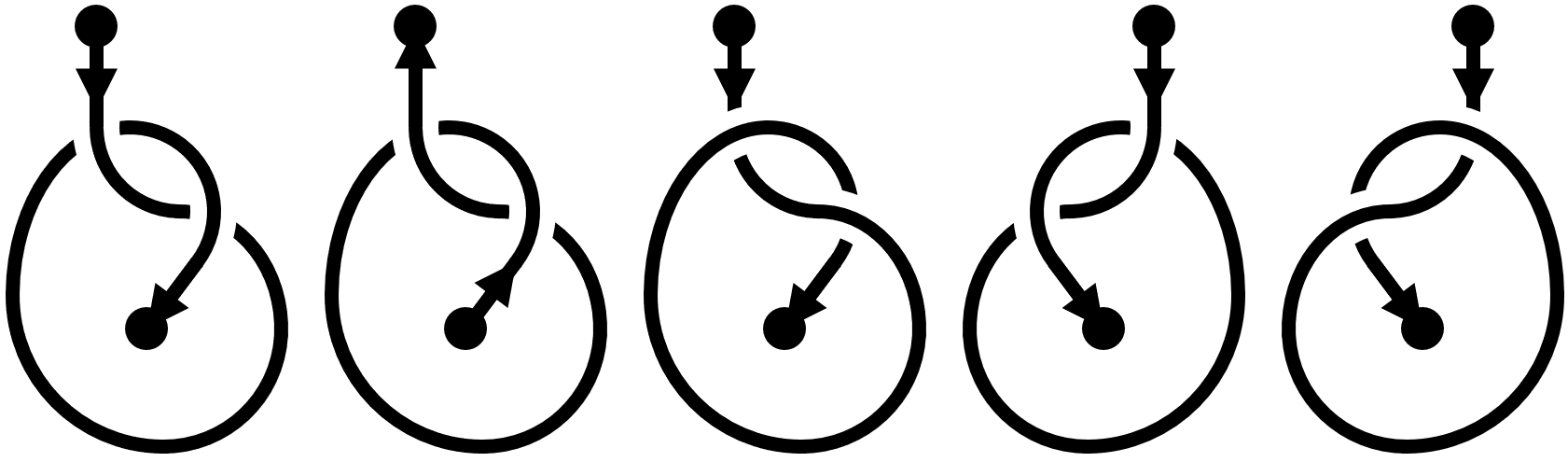}
\caption{From left to right, the bifoil $\varphi_1$, $\rev(\varphi_1)$, $\mir(\varphi_1)$, $\sym(\varphi_1)$, and $\rot(\varphi_1)$. Only the first two are equivalent.}
\label{fig:involutions}
\end{figure}

\subsection{Lifting}\label{sec:lifting}

Given a diagram $K$ of some knotoid $k$ and positive $n$, we may choose a lift of $K$ to the $n$-fold cover of $S^2$ branched over $v_0$ and $v_1$. The result is a new diagram $K/n$ with orientation and over/under information inherited from $K$, and $k/n$ is a well-defined knotoid. See \Cref{fig:lifting-fig}. For all $k$, the sequence $k/n$ stabilizes to a knot-type knotoid; we define $k/\infty$ to be the corresponding knot.

\begin{figure}
\includegraphics[scale=0.3]{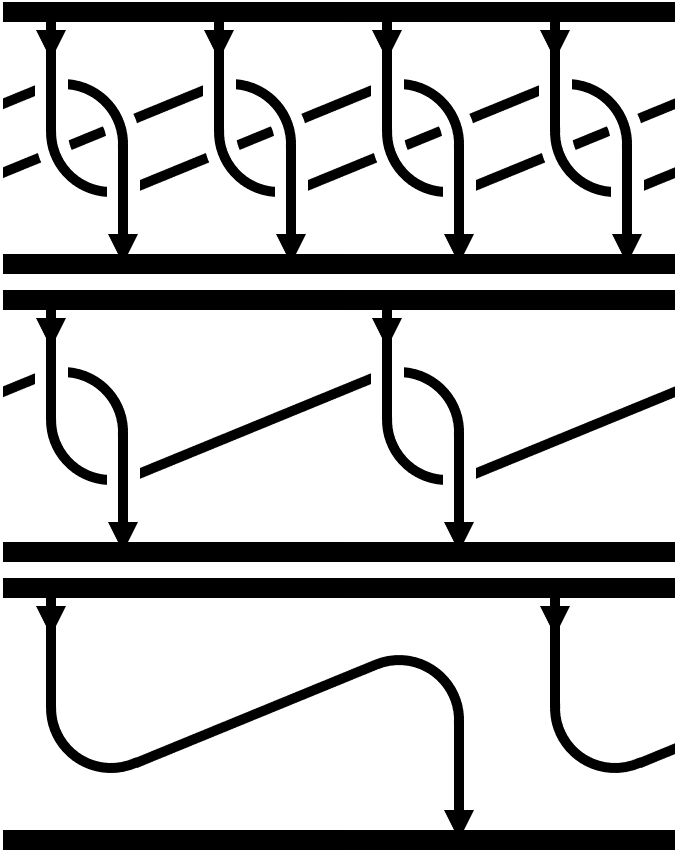}
\caption{Lifts of the spiral knotoid $\varphi_2$: $\varphi_2/1 = \varphi_2$, $\varphi_2/2 = \varphi_1$, and $\varphi_2/3 = (\varphi_2/\infty)^\bullet = 1$. Here we use \emph{periodic diagrams}: Given a knotoid diagram $K$ on $S^2$, we may delete regular neighborhoods of the endpoints to obtain a diagram on an annulus, with one end on each boundary component. Then on $\R \times I$ we draw the preimage of the diagram under the covering map.}
\label{fig:lifting-fig}
\end{figure}

A similar construction is studied in \cite{BBHL18}: The entire preimage of $K$ under the double cover of $S^2$ branched over the endpoints constitutes a diagram of an unoriented knot. The unoriented knot is an invariant and is called the \emph{double branched cover} of $k$.

\subsection{Framings}

A \emph{framing} of a knotoid $k$ is a class of diagrams of $k$ up to regular isotopy, that is, up to Reidemeister moves I', II, and III. Two diagrams of $k$ are in the same framing class if and only if they have the same writhe.

Similarly, a \emph{shortcut framing} for $k$ is a class of shortcut diagrams related by all moves except the Type I shortcut move. Shortcut framings are classified by algebraic intersection number between the main strand $K$ and the shortcut $a$.

\begin{figure}
\includegraphics[scale=0.2]{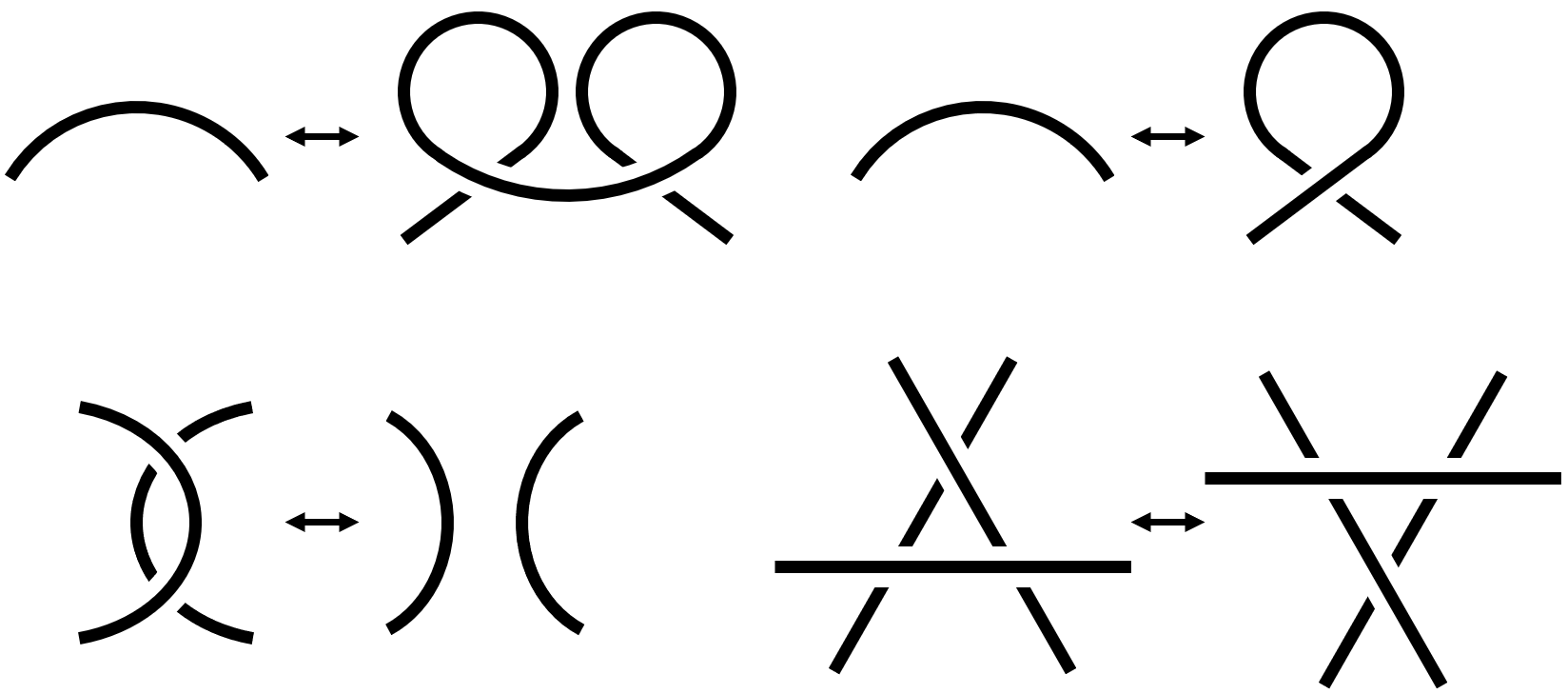}
\caption{Reidemeister moves, Types I', I, II, and III.}
\label{fig:reidemeister-moves}
\end{figure}

\subsection{Simple Theta-Curves}\label{sec:theta-curves}

The \emph{theta graph} $\Theta$ is the graph with two vertices, $v_0$ and $v_1$, and three oriented edges $e_0$, $e_+$, and $e_-$ from $v_0$ to $v_1$. A \emph{theta-curve} $\theta$ is an embedding of $\Theta$ into $S^3$. Such a curve is \emph{simple} if (the image of) $e_+ \cup e_-$ is unknotted. A \emph{spanning disk} $D$ for a simple theta-curve is a choice of embedded disk with boundary $\bdy D = e_+ \cup e_-$ such that $D$ intersects $e_0$ transversely. 

Like knotoids, isotopy classes of simple theta-curves form a monoid: The product of $\theta_1$ and $\theta_2$ is formed by deleting small balls around $v_1$ in $\theta_1$ and $v_0$ in $\theta_2$, and gluing the boundaries so that each edge of $\theta_1$ is glued to the edge of $\theta_2$ with the same label. The resulting theta-curve is well-defined up to isotopy, because the pure mapping class group of a thrice-punctured sphere is trivial. Similarly, we may multiply isotopy classes of pairs $(\theta, D)$.

The \emph{height} $h(\theta)$ of $\theta$ is the minimal number of intersections of a spanning disk with $e_0$, and the \emph{positive} and \emph{negative heights} $h_\pm$ are the minimal numbers of intersections of those signs. The \emph{sign sequence} associated to $(\theta, D)$ is the sequence of signs of the intersections of $e_0$ with $D$, in the order along $e_0$ from $v_0$ to $v_1$. Sequences obtained this way are \emph{attainable} for $\theta$.

\begin{figure}
\includegraphics[scale=0.15]{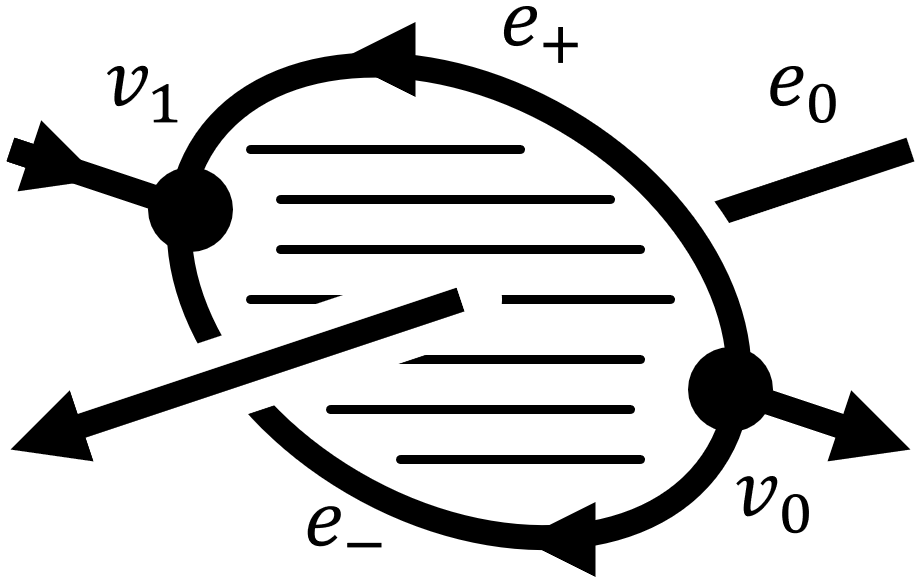}
\caption{A positive intersection of $e_0$ with a spanning disk.}
\label{fig:disk-sign-conv}
\end{figure}

There is a natural map $\tau$ from the set of knotoids to the set of isotopy classes of simple theta-curves. Given a shortcut diagram $(K, a)$, for a knotoid $k$, we may form a simple theta-curve by considering the diagram as lying in a neighborhood of $S^2 \subset S^3$, with $K$ lifting to an embedded path $e_0$, and adding edges $e_+$ and $e_-$ over and under $a$. The resulting theta-curve represents $\tau(k)$. Note that, in $\tau(k)$, the isotopy class of $e_0 \cup e_\pm$ is $k_\pm$.

\begin{thm}[Turaev \cite{Tur12}]
The map $\tau$ is a monoid isomorphism.
\end{thm}

The construction above also yields a correspondence between shortcut diagrams for $k$ and spanning disks for $\tau(k)$: For each such diagram $(K, a)$, we may choose $D$ to be the ``vertical" disk between $e_+$ and $e_-$ in the constructed theta-curve $\theta$. Under this construction, the sign sequence of $(\theta, D)$ is the same as that of $(K, a)$. Every isotopy class of spanned embeddings $(\theta, D)$ for $\tau(k)$ may be obtained in this way from a shortcut diagram for $k$. Thus, heights and attainable sign sequences for $k$ correspond directly with heights and attainable sequences for $\tau(k)$. Isotopy classes of spanned embeddings $(\theta, D)$ for $\tau(k)$ correspond to classes of shortcut diagrams of $k$ under planar isotopy, the Reidemeister moves away from $a$, and the Type III shortcut move.

\section{Comparing Attainable Sequences}\label{sec:moving-disks}

Here we prove Theorems \ref{thm:main-thm} and \ref{thm:shift-moves}. But first, we establish that allowing certain self-intersections in spanning disks would not reduce the height of a simple theta-curve.

\begin{lem}\label{lem:immersion-lemma}
Suppose, for some simple theta-curve $\theta$, that $\phi : \Delta \to S^3$ is an immersed (not necessarily embedded) disk such that (a) $\phi(\bdy\Delta) = e_+ \cup e_-$, (b) $e_0$ intersects $\phi$ transversely, and (c) the self-intersections of $\Delta$ under $\phi$ are disjoint circles identified transversely in pairs. Then $\phi(\Delta)$ has at least $h(\theta)$ intersections with $e_0$.
\end{lem}

\begin{proof}
There is some finite number of intersecting pairs of circles on $\Delta$. Any circle of self-intersection has a neighborhood in which $\phi$ looks like  the product of a plus sign with a circle. There are two ways of resolving the intersection by smoothing. The resolution that preserves the number of components creates a new immersion of a disk. Note that this resolution might not respect orientation. Replacing $\phi$ by this new immersion, we have reduced the number of self-intersections without changing the number of intersections with $e_0$. Proceeding in this fashion shows that there is a spanning disk with the same number of $e_0$ intersections as $\phi$.
\end{proof}

We will say that two spanning disks $D_1$ and $D_2$ for a theta-curve $\theta$ are \emph{compatible} if their interiors are disjoint. Compatible pairs of spanning disks are useful because of the following lemma.

\begin{lem}\label{lem:compatible-height}
If $D_1$ and $D_2$ are compatible spanning disks for a simple theta-curve $\theta$ and $h(\theta, D_1) \geq h(\theta, D_2)$, then $h_\pm(\theta, D_1) \geq h_\pm(\theta, D_2)$.
\end{lem}

\begin{proof}
Because $D_1$ and $D_2$ are compatible, $D_1 \cup D_2$ is an embedded sphere dividing $S^3$ into two balls. Note that $D_1 \cup D_2$ is not naturally oriented, because the orientations of $D_1$ and $D_2$ agree on $e_+ \cup e_-$. Let $\Sigma$ denote the oriented sphere $D_1 \cup (-D_2)$.

The net number of intersections of $e_0$ with $\Sigma$, not including $v_0$ or $v_1$, must be $0$, $1$, or $-1$. In symbols,
\begin{equation}\label{eq:net-in-out}
-1 \leq h_+(D_1) - h_-(D_1) - h_+(D_2) + h_-(D_2) \leq 1.
\end{equation}
(Here we have suppressed $\theta$ in the notation.)

By assumption, we have
\begin{equation}\label{eq:height-assumption}
h_+(D_1) + h_-(D_1) \geq h_+(D_2) + h_-(D_2).
\end{equation}
Combining (\ref{eq:height-assumption}) with each inequality in (\ref{eq:net-in-out}), we obtain $2h_+(D_1) - 2h_+(D_2) \geq -1$ and $2h_-(D_1) - 2h_-(D_2) \geq -1$. Therefore, $h_\pm(D_1) \geq h_\pm(D_2)$, as desired.
\end{proof}

The next fact provides opportunities to apply \Cref{lem:compatible-height}.

\begin{lem}\label{lem:main-lem-v1}
If $D$ and $D'$ are spanning disks for a simple theta-curve $\theta$ and $D$ realizes the height of $\theta$, then there is a sequence of spanning disks
$$D' = D_0, D_1, \ldots, D_n = D$$
such that consecutive disks are compatible and the sequence $(h(\theta, D_i))$ is nonincreasing.
\end{lem}

\begin{proof}
We may choose $D_1$ such that $(\theta, D_1)$ is isotopic to $(\theta, D')$ and such that $D$ and $D_1$ intersect transversely away from $C = e_+ \cup e_-$. Then the intersections consist of $C$ and a system $\calC_1$ of disjoint circles and arcs embedded properly in both $D$ and $D_1$. We may also require in our choice of $D_1$ that none of the intersection curves meet $e_0$. For each $i \geq 1$, once we have chosen $D_i$ we will form $D_{i+1}$ in such a way that $D_{i+1}$ has fewer total components of intersection with $D$ than does $D_i$. Let $\calC_i$ be the system of intersections between $D_i$ and $D$, not including $C$.

\case{Case 1: $\calC_i$ is empty.}

If $\calC_i$ is empty, then $D$ is compatible with $D_i$, so we set $n = i + 1$ and $D_n = D$. Because $D$ realizes the height of $\theta$, $h(D_i) \geq h(D)$.

\case{Case 2: $\calC_i$ has an arc, but no circles.}

If $\calC_i$ contains an arc, but no circles, we can find an innermost such arc $s$ on $D$. By innermost arc, we mean one for which all other arcs lie on one side of $s$ in $D$. In particular, the endpoints of $s$ divide $C$ into two segments $t$ and $t'$ such that the endpoints of all other curves of $\calC_i$ lie on $t'$. Let $E$ be the disk in $D$ bounded by $S = s \cup t$, and $E'$ the disk in $D_i$ bounded by $S$. Since $t$ contains no endpoints of the arcs in $\calC_i$, $s$ is also innermost in $D_i$. In particular, $E \cap D_i = E' \cap D = S$.

Since $(D - E) \cup E'$ is an embedded disk with boundary $C$, its height is at least $h(D)$. Therefore, $E'$ has at least as many intersections with $e_0$ as does $E$. Now let $D_{i+1}$ be the result of slightly perturbing $(D_i - E') \cup E$ to be compatible with $D_i$. Then $D_{i+1}$ has fewer intersection curves than $D_i$ with $D$, and $h(D_{i+1}) \leq h(D_i)$.

\case{Case 3: $\calC_i$ contains a circle.}

If there is at least one circle, there is an innermost circle $S$ in $D$. Then $S$ bounds a disk $E \subset D$ with $E \cap D_i = S$. Let $E'$ denote the disk in $D_i$ bounded by $S$. In contrast with Case 2, $S$ is not necessarily innermost in $D_i$, so $E' \cap D$ may be more than just $S$.

By \Cref{lem:immersion-lemma}, $e_0$ has at least as many intersections with $(D - E) \cup E'$ as with $D$, so it intersects $E'$ at least as many times as $E$. Therefore, we may proceed as in Case 2. Let $D_{i+1} = (D_i - E') \cup E$, and perturb it so that it is compatible with $D_i$. Then $h(D_{i+1}) \leq h(D_i)$, and we have reduced the number of intersection curves with $D$.

\casesend

This covers all the cases, so we are done.
\end{proof}

\Cref{lem:compatible-height} immediately implies that \Cref{lem:main-lem-v1} can be strengthened as follows.

\begin{lem}\label{lem:main-lem}
If $D$ and $D'$ are spanning disks for a simple theta-curve $\theta$ and $D$ realizes the height of $\theta$, then there is a sequence of spanning disks
$$D' = D_0, D_1, \ldots, D_n = D$$
such that consecutive disks are compatible and the sequences $(h_\pm(\theta, D_i))$ are both nonincreasing.
\end{lem}

We can now prove Theorems \ref{thm:main-thm} and \ref{thm:shift-moves}.

\begin{proof}[Proof of \Cref{thm:main-thm}]
Given a simple theta-curve $\theta$, a spanning disk $D$ realizing the height, and any other spanning disk $D'$, \Cref{lem:main-lem} implies that $D$ has no greater positive or negative height than $D'$. Therefore, $D$ realizes the signed heights, and so,
$$h(\theta) = h(\theta, D) = h_+(\theta, D) + h_-(\theta, D) = h_+(\theta) + h_-(\theta).$$

For a knotoid $k$, we obtain $h(k) = h_+(k) + h_-(k)$ by setting $\theta = \tau(k)$.
\end{proof}

To prove \Cref{thm:shift-moves}, we will use another lemma about compatible spanning disks.

\begin{lem}\label{lem:compatible-shift-moves}
If $D_1$ and $D_2$ are compatible spanning disks for a simple theta-curve $\theta$ and $h(\theta, D_1) = h(\theta, D_2)$, then $\Seq(\theta, D_1)$ and $\Seq(\theta, D_2)$ differ by a shift move.
\end{lem}

\begin{proof}
Let $\Sigma$ be the sphere $D_1 \cup (-D_2)$ as in the proof of \Cref{lem:compatible-height}. Let $B$ be the ball in $S^3$ such that $\Sigma$ is the oriented boundary of $B$, and assume that $e_0 - \{v_0, v_1\}$ ``starts" outside of $B$. In the overall sequence of intersections of $e_0$ with $\Sigma$ (not including the endpoints), the signs of the intersections alternate: The odd- and even-index intersections are negative and positive, respectively. Because $h(\theta, D_1) = h(\theta, D_2)$, the total number of intersections is even. Each odd-even pair of consecutive intersections has type $(+D_2, +D_1)$, $(-D_1, -D_2)$, $(-D_1, +D_1)$, or $(+D_2, -D_2)$. Therefore, $\Seq(\theta, D_2)$ is obtained from $\Seq(\theta, D_1)$ by a left shift move.

If $e_0$ instead starts in the inside of $B$, then $\Seq(\theta, D_2)$ is obtained from $\Seq(\theta, D_1)$ by a right shift move.
\end{proof}

\begin{proof}[Proof of \Cref{thm:shift-moves}]
For a simple theta-curve $\theta$, suppose $D$ and $D'$ are both spanning disks realizing the height of $\theta$. Then in the sequence $(D_i)$ of spanning disks obtained from \Cref{lem:main-lem}, each $D_i$ realizes the height of $\theta$. Therefore, applying \Cref{lem:compatible-shift-moves} to the sequence $(D_i)$ implies that $\Seq(\theta, D)$ and $\Seq(\theta, D')$ are connected among minimal attainable sequences by shift moves.
\end{proof}

\section{Knotoid Operations and Attainable Sequences}\label{sec:properties}

\subsection{Signed Heights under the Basic Involutions}

The signed heights of knotoids behave in straightforward ways under the basic knotoid involutions.

\begin{prop}\label{prop:signed-ht-involutions}
For all $k$, we have
$$h_\pm (k) = h_\pm(\rev(k)) = h_\pm(\mir(k)) = h_\mp(\sym(k)) = h_\mp(\rot(k)).$$
\end{prop}

More specifically, we can say the following.

\begin{prop}\label{prop:involutions-sequences}
If $A$ is an attainable sign sequence for $k$, then,
\begin{enumerate}
\item $\rev(A)$ is attainable for $\rev(k)$,
\item $A$ is attainable for $\mir(k)$,
\item $-A$ is attainable for $\sym(k)$, and
\item $-A$ is attainable for $\rot(k)$.
\end{enumerate}
where $-A$ is the result of switching all terms $+ \leftrightarrow -$ in $A$ and $\rev(A)$ is the result of reversing the order.
\end{prop}

\begin{proof}
Given a shortcut diagram $(K, a)$ for $k$, the signs of the crossings between $K$ and $a$ are not changed by switching the over/under information of $K$ or by simultaneously switching the orientations of $K$ and $a$ (recall that the orientation for a shortcut is determined by the rest of the diagram). However, changing the orientation on $S^2$ changes the signs of the intersections.
\end{proof}

\Cref{prop:signed-ht-involutions}, together with \Cref{thm:main-thm}, has implications for unsigned heights of knotoids, such as for rotatable knotoids, which are addressed in \cite{BBHL18}. A knotoid $k$ is \emph{rotatable} if it equals $\rot(k)$.

\begin{cor}\label{cor:odd-rotatable}
Every rotatable knotoid has even height.
\end{cor}

\subsection{Multiplication and Concatenation}

In this section we relate the set of attainable sign sequences for a product to the attainable sequences of its factors.

\begin{prop}\label{prop:seq-concatenation-d1}
For any $k_1$ and $k_2$, if $A_1$ is an attainable sign sequence for $k_1$ and $A_2$ is attainable for $k_2$, then the concatenation $A_1A_2$ is attainable for $k_1k_2$.
\end{prop}

\begin{proof}
Given shortcut diagrams for $k_1$ and $k_2$, the sign sequence for the product of the diagrams is the concatenation of the sign sequences for the two original diagrams.
\end{proof}

Note that as a particular case of the statement above, if $A$ is any attainable sequence for a knotoid $k$, then the result of appending $+$ or $-$ to either end of $A$ is also attainable for $k$, because on any shortcut diagram for $k$ we can perform a Type I shortcut move around either endpoint.

The next theorem is a converse for \Cref{prop:seq-concatenation-d1}.

\begin{thm}\label{thm:seq-concatenation-d2}
Any minimal attainable sequence for $k_1k_2$ is the concatenation of minimal attainable sequences for $k_1$ and $k_2$.
\end{thm}

We will prove \Cref{thm:seq-concatenation-d2} using a modification of original argument appearing in \cite{Tur12} for \Cref{cor:height-additivity} below, which is an immediate corollary.

\begin{cor}
For two knotoids $k_1$ and $k_2$, $h_\pm(k_1k_2) = h_\pm(k_1) + h_\pm(k_2)$.
\end{cor}

\begin{cor}[Turaev \cite{Tur12}]\label{cor:height-additivity}
For two knotoids $k_1$ and $k_2$, $h(k_1k_2) = h(k_1) + h(k_2)$.
\end{cor}

We will treat knotoids as simple theta-curves, and use the following lemmas:

\begin{lem}\label{lem:extra-spheres}
Suppose that $\theta$ is a simple theta-curve and $\Delta$ is a compact oriented surface (not necessarily connected) embedded in $S^3$ such that (a) $\bdy\Delta = e_+ \cup e_-$, (b) $e_0$ intersects $\Delta$ transversely, and (c) The component of $\Delta$ with boundary is a disk. Then the sequence $\Seq(\theta, \Delta)$ of signs of intersections of $e_0$ with $\Delta$ is attainable for $\theta$.
\end{lem}

\begin{proof}
Let $D$ be the disk component of $\Delta$.

\case{Case 1: Every closed component of $\Delta$ is a sphere.}

If $\Delta = D$, then of course $\Seq(\theta, \Delta) = \Seq(\theta, D)$ is attainable.

If $\Delta$ is not connected, some spherical component $H$ of $\Delta$ must be ``outermost" in the sense that no other sphere separates it from $D$. If $D$ is on the positive side of $H$, then we may choose an embedded path from $H$ to the positive side of $D$ such that the path does not otherwise intersect $\Delta$ or $e_0$. Then we can incorporate $H$ into $D$ in an orientation-respecting way by adding an annulus to connect $H$ to $D$ and deleting disks in $D$ and $H$ around the path's endpoints. If $D$ is on the negative side of $H$, we do the same but with the negative side of $D$. Doing this several times replaces $\Delta$ with a spanning disk and realizes the sign sequence as attainable.

\case{Case 2: General Case.}

Each closed component of $\Delta$ separates $S^3$, and as in Case $1$ we can consider $D$ to be ``outside" of every other component, regardless of their orientations. Call a spherical component of $\Delta$ \emph{trivial} if it intersects $e_0$ twice and $e_0$ is unknotted inside the sphere. Let $N$ be the set of closed components that are not trivial spheres. If $N$ is nonempty, consider an innermost element $G$ of $N$. Inside of $G$ are some number of segments of $e_0$. Some of these segments may have trivial spheres attached. Let us delete $G$ and replace it with several trivial spheres: One sphere is added for each segment of $e_0$ inside $G$, surrounding the segment and all preexisting trivial spheres on that segment. The new spheres may be oriented appropriately so that we have not changed $\Seq(\theta, \Delta)$. Repeating this process renders $N$ empty and reduces us to Case 1.
\end{proof}

\begin{lem}\label{lem:knot-type-lem}
For any simple theta-curve $\theta$ and knot-type theta-curve $\kappa^\bullet$, a sign sequence is attainable for $\kappa^\bullet\theta$ if and only if it is attainable for $\theta$.
\end{lem}

\begin{proof}
It is immediate that any attainable sequence for $\theta$ is attainable for $\kappa^\bullet\theta$.

For the other direction, suppose we have a spanning disk $D$ for $\kappa^\bullet\theta$. Pick a ball $B$ such that (a) $B$ intersects $\kappa^\bullet\theta$ only on $e_0$, (b) $\bdy B$ intersects $D$ transversely, and (c) the restriction of $e_0$ to $B$ is $\kappa$ in the form of a $1$-tangle. If $e_0$ intersects $D$ inside $B$, we may push these intersections to the outside: Choose a subinterval of $e_0 \cap B$ containing the intersections with $D$ as well as one of the two endpoints, then delete from $B$ a regular neighborhood of that interval. Therefore we may choose $B$ so that $e_0 \cap D \cap B$ is empty.

Let $b_0$ and $b_1$ be the intersections of $e_0$ with $\bdy B$, assigned such that $e_0$ is oriented from $b_0$ to $b_1$. Each component of $D \cap B$ is a genus-$0$ surface (with boundary) properly embedded in $B$, and each component of $D \cap \bdy B$ is an oriented circle with winding number $0$, $1$, or $-1$ around $\bdy B - \{b_0, b_1\}$. Suppose there is at least one circle with winding number $0$. Then there is an innermost such circle. We may cut $D$ along this circle and fill in two disks on either side of $\bdy B$ to obtain a new spanning surface consisting of a disk and a sphere. Doing these repeatedly, we obtain a surface $\Delta$ as in \Cref{lem:extra-spheres} (specifically, as in Case 1) such that $\Seq(\kappa^\bullet\theta, \Delta) = \Seq(\kappa^\bullet\theta, D)$ and and such that every component of $\Delta \cap \bdy B$ has winding number $\pm 1$. There may now be some spherical components of $\Delta$ contained entirely within $B$, but they do not intersect $e_0$.

Now, let us label the components of $\Delta \cap B$ as $E_1, \ldots, E_n$. Each separates $B$ into two regions, and each is disjoint from the others and from $e_0$. The components of $\Delta \cap \bdy B$, all concentric circles, have an order based on how they are arranged from $b_0$ to $b_1$ and so can be indexed $1, \ldots, m$. For each $i$, let $s(i)$ be the winding number of the $i$\supscr{th} circle and let $c(i)$ be the index of its component in $\Delta \cap B$. For each $j \in \{1, \ldots, n\}$, the sum of the $s(i)$ over all $i$ with $c(i) = j$ must be $0$, because $e_0$ does not intersect $E_j$. For similar reasons, for all $j$ and all $w_1$ and $w_2$ with $w_1 < w_2$ and $c(w_1) = c(w_2) \neq j$, the sum of the $s(i)$ over all $i$ with $c(i) = j$ and $w_1 < i < w_2$ is also $0$.

Now, to show that $\Seq(\kappa^\bullet\theta, \Delta)$ is attainable for $\theta$, we will create a spanning of $\theta$ by deleting and replacing the interior of $B$. Let $B'$ be a standard $3$-ball, and choose an orientation-respecting identification $\bdy B' \cong \bdy B$. Let $e_0'$ be an unknotted strand properly embedded in $B'$ from $b_0$ to $b_1$. Now consider a partition of the components of $\Delta \cap \bdy B$ into pairs such that (a) paired components have opposite winding numbers, (b) paired components come from the same component of $\Delta \cap B$, and (c) for $i_1 < i_2 < i_3 < i_4$, we do not have $i_1$ paired to $i_3$ and $i_2$ to $i_4$. (A simple induction argument shows this is possible.) Now, we connect each pair of components with an unknotted annulus disjoint from $e_0'$ and disjoint from the other annuli. See \Cref{fig:filling-sphere}.

\begin{figure}
\includegraphics[scale=0.4]{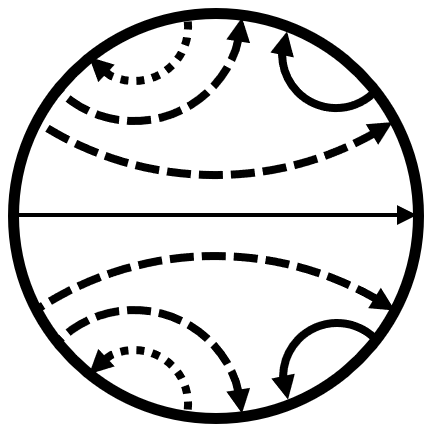}
\caption{A cross section of $B'$. We can obtain $B'$ by rotating around the center axis, which represents $e_0'$. In this example, $\Delta \cap B$ had three components, indicated by the three types of curved arrow.}
\label{fig:filling-sphere}
\end{figure}

Now we glue $B'$ along $\bdy B$ to the closure of the complement of $B$ to obtain a spanning surface $\Delta'$ for $\theta$, with $\Seq(\theta, \Delta') = \Seq(\kappa^\bullet\theta, D)$. Because we replaced each component of $\Delta \cap B$ with several annuli, we have not created any higher-genus components by replacing $\Delta$ with $\Delta'$. Therefore $\Delta'$ is a union of a disk with spheres, and so by \Cref{lem:extra-spheres}, $\Seq(\kappa^\bullet\theta, D)$ is attainable for $\theta$.
\end{proof}

\begin{proof}[Proof of \Cref{thm:seq-concatenation-d2}]
Given a spanning disk $D$ for a product $\theta_1\theta_2$ realizing its height, we wish for there to be an embedded sphere decomposing $(\theta_1\theta_2, D)$ as a product of spannings of $\theta_1$ and $\theta_2$. Such a sphere does not exist in general, but we will assume that $\theta_1$ and $\theta_2$ have no knot-type factors, and this will be sufficient by \Cref{lem:knot-type-lem}.

By construction, there is a sphere $\Sigma$ decomposing $\theta_1\theta_2$ as a product of $\theta_1$ and $\theta_2$. Necessarily, the two vertices of $\theta_1\theta_2$ lie on opposite sides of $\Sigma$, and each edge intersects $\Sigma$ once transversely. We may assume $D$ to intersect $\Sigma$ transversely as well. Then the intersection of $D$ and $\Sigma$ consists of a line segment and possibly several circles. If the number of circles is $0$, then $\Sigma$ cuts $D$ into two disks, which are spanning disks for $\theta_1$ and $\theta_2$, so we are done.

If there are some circles, we may pick one which is innermost in $D$. This bounds a disk $E$ in $D$ which does not otherwise intersect $D$ or $\Sigma$. It also separates $\Sigma$ into two disks $\Sigma_1$ and $\Sigma_2$. Since $E$ sits on one side of $\Sigma$ and cuts that side into two parts, one part contains a vertex and the other does not. We may assign the labels $\Sigma_1$ and $\Sigma_2$ in such a way that $\Sigma_1 \cup E$ is a sphere that separates the vertices of $\theta$ and $\Sigma_2 \cup E$ is a sphere with both vertices on one side. Let $B$ denote the ball with boundary $\Sigma_2 \cup E$ that doesn't contain the vertices.

Each of $e_+$ and $e_-$ must have its one intersection with $\Sigma$ on $\Sigma_1$, as it cannot intersect $E$. The intersection of $e_0$ with $\Sigma$ may be on either $\Sigma_1$ or $\Sigma_2$, but regardless, $e_0$ cannot intersect $E$ more times than it intersects $\Sigma_2$, by \Cref{lem:immersion-lemma}. Therefore, $e_0$ either intersects $\Sigma_1$ once and not $\Sigma_2$ or $E$, or it intersects $\Sigma_2$ and $E$ once each but not $\Sigma_1$. In the latter case, there is a $1$-tangle inside of $B$, but by our assumption of no knot-type factors, the tangle is unknotted.

Let $\Sigma'$ be the sphere formed by pushing $\Sigma_2$ through $B$ and past $E$, so $\Sigma'$ is a slight perturbation of $\Sigma_1 \cup E$ and there are fewer circular intersections of $D$ with $\Sigma'$ than with $\Sigma$. Since $B$ either does not intersect $e_0$ or contains only an unknotted segment between $\Sigma_2$ and $E$, $(\theta, \Sigma')$ is isotopic to $(\theta, \Sigma)$, so $\Sigma'$ still decomposes $\theta$ as $\theta_1\theta_2$.

Repeating the above steps yields a sphere intersecting $e_0$ once and $D$ in only an interval, so it decomposes $(\theta_1\theta_2, D)$ as a product of $(\theta_1, D_1)$ and $(\theta_2, D_2)$ as desired. Then $\Seq(\theta_1\theta_2, D)$ is the concatenation of $\Seq(\theta_1, D_1)$ with $\Seq(\theta_2, D_2)$. Since $\Seq(\theta_1\theta_2, D)$ is minimal, each  $\Seq(\theta_i, D_i)$ is also minimal.
\end{proof}

\subsection{Signed Heights under Lifting}

For a shortcut diagram $(K, a)$ of a knotoid $k$, there are $n$ lifts of $a$ to a shortcut for $K/n$. The total number of positive/negative intersections of $K/n$ with all such lifts is equal to $h_\pm(K, a)$. Of course, that amount must be at least $n$ times the minimal number of positive/negative intersections with each of the $n$ lifts of $a$.

\begin{prop}
For all $k$ and $n$, $nh_\pm(k/n) \leq h_\pm(k)$.
\end{prop}

Furthermore, we can obtain attainable sign sequences for $k/n$ from attainable sequences for $k$ in the following way. Given a sign sequence $A$ of length $r$, for each $i \in \{0, \ldots, r\}$ let $p_A(i)$ be the sum of the terms of $A$ from indices $1$ to $i$. For $i \in \{1, \ldots, r\}$, let $q_A(i)$ be the maximum of $p_A(i - 1)$ and $p_A(i)$. Then for $x \in \Z/n\Z$, let $A^x$ be the subsequence of $A$ consisting of only the terms from indices $i$ with $q_A(i) \equiv x \pmod n$. Given $(K, a)$, we may label the $n$ lifts of $a$ as $a^1, \ldots, a^n$ in such a way that they increment counterclockwise around the lift of $v_0$, and the initial direction of $K/n$ is between $a^n$ and $a^1$. Then the $i$\supscr{th} intersection of $K$ with $a$ lifts to an intersection of $K/n$ with $a^{q(i)}$, so for each $x$, $\Seq(K/n, a^x) = \Seq(K, a)^x$. This implies the following.

\begin{prop}
For every attainable sequence $A$ for $k$, each $A^x$ is attainable for $k/n$.
\end{prop}

\section{Bounds on Signed Height}\label{sec:poly-bounds}

\subsection{Writhes}

Given a crossing $c$ in a knotoid diagram $K$, there is a unique resolution of $c$ that respects orientation. This resolution creates an oriented diagram with two components, a loop $L$ and an interval $K'$ with the same endpoints as $K$. The winding number of $L$ around the twice-punctured sphere is called the \emph{intersection index} of $c$, denoted $\Ind(c)$. The index is equal to the intersection number of $L$ with any shortcut, or with $K'$. Note that the index of a crossing doesn't depend on any over/under information. If a crossing has index $n$, it will be called an \emph{$n$-crossing}.

\begin{defn}\label{defn:writhe-defn}
For nonzero $n$, the \emph{$n$-writhe} $J_n(K)$ of $K$ is half the sum of the signs of the $n$-crossings.
\end{defn}

\begin{thm}[Kim--Im--Lee \cite{KIL18}]
For nonzero $n$, the $n$-writhe is a knotoid invariant.
\end{thm}

\begin{rem}
Our convention differs from \cite{KIL18} by a factor of $2$; they omit the word ``half" in \Cref{defn:writhe-defn}. Under our convention, the $n$-writhe is still an integer: Any knotoid diagram can be turned into a diagram for the trivial knotoid by switching the signs of crossings such that each ``late" strand passes over each ``early" strand. Each such switch changes the $n$-writhe by an integer, and the $n$-writhe of the trivial knotoid is $0$, so all $n$-writhes of all knotoids are integers. However, what we say here does not apply in general to virtual knotoids, which are considered in \cite{KIL18} alongside classical knotoids.
\end{rem}

The following are immediate consequences of the definition of $n$-writhe:

\begin{prop}
For a knotoid $k$, we have the following:
\begin{enumerate}
\item $J_n(\rev(k)) = J_n(k)$
\item $J_n(\mir(k)) = -J_n(k)$
\item $J_n(\sym(k)) = -J_{-n}(k)$
\item $J_n(\rot(k)) = J_{-n}(k)$
\end{enumerate}
\end{prop}

\begin{prop}
For knotoids $k_1$ and $k_2$, $J_n(k_1k_2) = J_n(k_1) + J_n(k_2)$.
\end{prop}

The $n$-writhes of a knotoid can be encoded in the coefficients of a polynomial. The \emph{index polynomial} for $k$ is
$$F_k(t) = \sum_{n\neq0} J_n(k)(t^n - 1) \in \Z[t, t^{-1}].$$
This is closely related to its similarly-named predecessor, the \emph{affine index polynomial} of \cite{GuKa17}, defined by
$$P_k(t) = \sum_c \sign(c)(t^{w(c)} - 1),$$
where $w(c)$ is $\sign(c)\ssgn(c)\Ind(c)$, and $\ssgn(c)$ is as shown in \Cref{fig:sequential-sign-con}. Note that $w(c)$ differs from $\Ind(c)$ only by sign. The affine index polynomial satisfies $P_k(t) = P_k(t^{-1})$ for all $k$ (\cite{GuKa17}), so it is related to the index polynomial by the formula
\begin{equation}\label{eq:AIP-formula}
P_k(t) = F_k(t) + F_k(t^{-1}).
\end{equation}
The $t^{\pm n}$ coefficient $J_n(k) + J_{-n}(k)$ of the affine index polynomial of $k$ equals the natural linking number of consecutive components in a periodic diagram for $k/n$ (see \Cref{fig:lifting-fig}).

\begin{figure}
\includegraphics[scale=0.2]{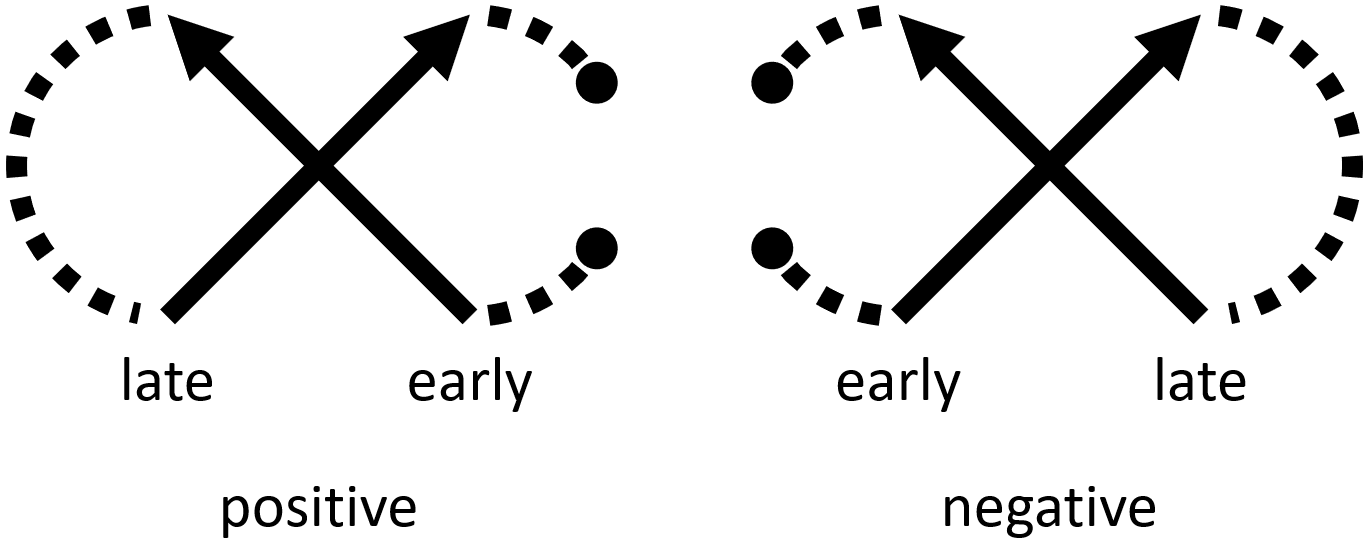}
\caption{Sequential signs of crossings.}
\label{fig:sequential-sign-con}
\end{figure}

The degree of the affine index polynomial was shown to be a lower bound for the height of a knotoid in \cite{GuKa17}. Because of the relationship in \cref{eq:AIP-formula}, this is equivalent to Proposition 3.12 of \cite{KIL18}. \Cref{thm:writhe-bound}, together with \Cref{thm:main-thm}, is an improvement on this bound in the case that $\deg^+(F_k)$ and $\deg^-(F_k)$ are both positive.

We now prove Theorems \ref{thm:writhe-bound}, \ref{thm:unique-sequences}, and \ref{thm:writhe-sequences}.

\begin{proof}[Proof of \Cref{thm:writhe-sequences}]
For nonzero $n$, if $J_n(k) \neq 0$, any shortcut diagram for $k$ must have an $n$-crossing $c$. Then the segment of $K$ starting and ending at $c$ has, algebraically, $n$ intersections with the shortcut $a$, so the signs in the corresponding segment of $\Seq(K, a)$ add up to $n$.
\end{proof}

\begin{proof}[Proof of \Cref{thm:writhe-bound}]
If $J_n(k) \neq 0$, then as above, in any attainable sign sequence for $k$ there is a consecutive subsequence with sum $n$. Therefore, for positive $n$ there must be at least $n$ appearances of $+$, and for negative $n$ there are at least $-n$ appearances of $-$. This proves that the positive/negative height of $k$ is bounded below by the positive/negative degree of $F_k$.
\end{proof}

\begin{proof}[Proof of \Cref{thm:unique-sequences}]
Suppose that $h_\pm(k) = \deg^\pm(F_k)$. A minimal attainable sign sequence contains $h_+(k)$ copies of $+$ and $h_-(k)$ copies of $-$, and by \Cref{thm:writhe-sequences}, the terms of the same sign must all be consecutive. Therefore, any minimal attainable sign sequence is one of $(+, \ldots, +, -, \ldots, -)$ or $(-, \ldots, -, +, \ldots, +)$. Call these two sequences $A_1$ and $A_2$, respectively. To show that only one of these can be attainable, we consider several cases.

\case{Case 1: $h_+(k)$ or $h_-(k)$ is $0$.}

If one of the signed heights is zero, then all of the terms are the same sign, and $A_1 = A_2$.

\case{Case 2: $h_+(k), h_-(k) > 1$.}

In this case, $A_1$ and $A_2$ are not related by a shift move. There are no other minimal attainable sequences, so by \Cref{thm:shift-moves}, they cannot both be attainable.

\case{Case 3: $h_+(k)$ or $h_-(k)$ is $1$ and neither is $0$.}

Supposing that $A_1$ and $A_2$ are both attainable, they are the only minimal attainable sequences. By \Cref{lem:main-lem}, there are compatible spanning disks $D_1$ and $D_2$ for $\tau(k)$ such that $\Seq(\tau(k), D_i) = A_i$ for $i = 1, 2$.

Since $D_1 \cup D_2$ is an embedded sphere in $S^3$ and $e_+ \cup e_-$ is an embedded circle in $D_1 \cup D_2$, we may pick an embedded sphere $\Sigma$ such that $\Sigma$ intersects $e_+ \cup e_-$ at $v_0$ and $v_1$ only, and such that $D_1$ and $D_2$ each intersect $\Sigma$ in an interval. Then we may slide $e_0$ down onto $\Sigma$ and obtain a diagram $K$ for $k$ with two \emph{compatible} shortcuts $a_1$ and $a_2$ corresponding respectively to $D_1$ and $D_2$ (see \Cref{fig:two-shortcut}). We have a region $E$ in $S^2$ bounded by $a_1 \cup a_2$ such that there are $h(k)$ segments of $K$ in $E$, one entering and leaving by $a_1$, one entering and leaving by $a_2$, and the rest crossing from one side to the other. Without loss of generality, we may assume that, starting from $v_0$, $K$ intersects $a_2$ before $a_1$. Then $h_-(k)$ must be $1$, and the intersections come in the order
$$(-a_2, +a_2, \ldots, +a_1, -a_1),$$
where the ``$\ldots$" consists of $h_+(k) - 1$ consecutive copies of $(+a_1, +a_2)$.
Since any crossing $c$ of $K$ lies either in $E$ or the complement of $E$, the loop on $K$ from $c$ to $c$ has an even total number of intersections with $a_1$ and $a_2$. Therefore, if that loop includes the negative intersection with $a_1$, it also includes at least one positive $a_1$ intersection, so if we measure the index of $c$ by intersections of the loop with $a_1$, the index is nonnegative.

\begin{figure}
\includegraphics[scale=0.4]{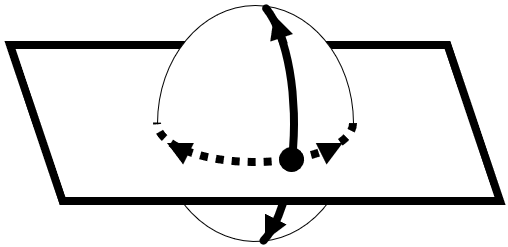}
\caption{Compatible spanning disks correspond to compatible shortcuts.}
\label{fig:two-shortcut}
\end{figure}

Since all crossings have nonnegative index, all negative writhes are $0$, contradicting the assumption that $h_-(k) = \deg^-(F_k)$. This proves \Cref{thm:unique-sequences}.
\end{proof}

\subsection{The Turaev Polynomial}\label{sec:turaev-poly}

A \emph{state} of a diagram $K$ is a function from the set of crossings to $\{-, +\}$. For each state $s$, the $s$-smoothing of $K$ is given by smoothing each crossing according to \Cref{fig:bracket-sign}. The sum of $s(c)$ over all crossings is denoted $n(s)$, and the number of embedded circles in the diagram after smoothing by $s$ is $\ell(s)$. (There is also one embedded interval, which is not counted.) The \emph{bracket polynomial} of $K$ is then
$$\la K\ra = \sum_s A^{n(s)}(-A^2 - A^{-2})^{\ell(s)} \in \Z[A, A^{-1}].$$
The bracket polynomial is invariant under Reidemeister moves I', II, and III, so it is a framed knotoid invariant. A Reidemeister I move changes the bracket polynomial by a factor of $-A^{-3}$, so the \emph{normalized bracket polynomial} defined by
$$\la K\ra_\circ = (-A)^{-3\wri(K)}\la K\ra$$
in \cite{Tur12} is an unframed invariant.

\begin{figure}
\includegraphics[scale=0.25]{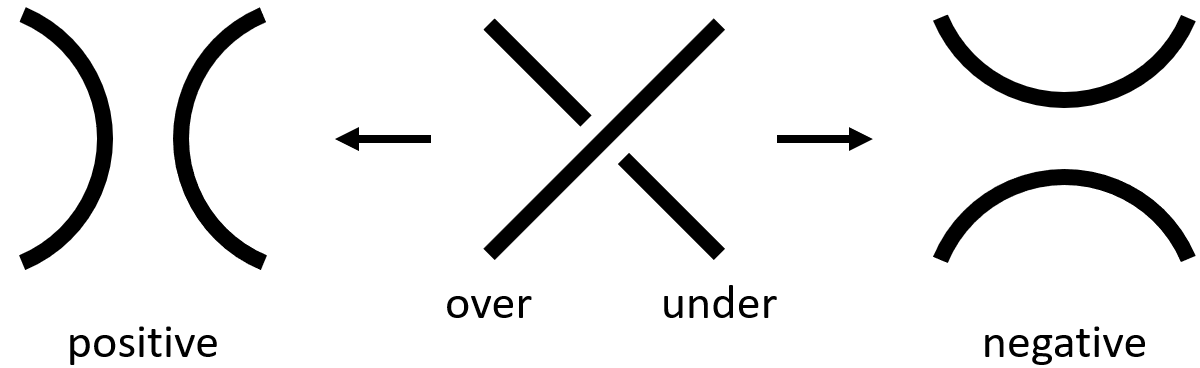}
\caption{Smoothings in the bracket polynomial.}
\label{fig:bracket-sign}
\end{figure}

There is also a two-variable version of the bracket polynomial, called the \emph{extended bracket polynomial} in \cite{Tur12} or the \emph{Turaev polynomial} as in \cite{Kut20}. For a shortcut diagram $(K, a)$, let $a(K)$ denote the algebraic height
$$h_+(K, a) - h_-(K, a),$$
and for any state, let $a(s)$ be the algebraic height of the interval component of the $s$-smoothing of $K$, with its natural orientation. Then the Turaev polynomial of $(K, a)$ is
$$\laa K, a\raa = \sum_s A^{n(s)}u^{a(s)}(-A^2 - A^{-2})^{\ell(s)} \in \Z[A^{\pm1}, u^{\pm1}].$$
This is an invariant of knotoids with both a framing and shortcut framing. The normalized version
$$\laa K, a\raa_\circ = (-A)^{-3\wri(K)}u^{-a(K)}\laa K, a\raa$$
is a knotoid invariant and always takes values in $\Z[A^{\pm2}, u^{\pm2}]$.

The height of a knotoid satisfies $2h(k) \geq \deg_u^+(\laa k\raa_\circ) + \deg_u^-(\laa k\raa_\circ)$ (\cite{Tur12}). \Cref{thm:turaev-bound} does not improve this bound on the overall height but is the equivalent statement for the signed heights.

\begin{proof}[Proof of \Cref{thm:turaev-bound}]
Fix a shortcut diagram $(K, a)$ representing a knotoid $k$. For any state $s$, the $s$-smoothing of $K$ only has as many intersections with $a$ as $K$ does. In particular, the interval component of the smoothing has no more than
$$h_+(K, a) + h_-(K, a)$$
positive or negative intersections with $a$, so we have 
$$-2h_+(K, a) \leq a(s) - a(K) \leq 2h_-(K, a).$$
Therefore, the $u$ exponents of $\laa k\raa_\circ$ are no more than $2h_-(k)$ and no less than $-2h_+(k)$.
\end{proof}

A categorification of the Turaev polynomial, the triply-graded \emph{winding homology}, is defined in \cite{Kut20}. The corresponding Poincar\'e polynomial is denoted $W_k(t, A, u)$ and satisfies
$$W_k(-1, A, u) = \laa k\raa_\circ$$
for every $k$. The winding homology is the homology of a chain complex in which each generator is given a $u$-grading $a(s) - a(K)$ for some state $s$, so in addition to \Cref{thm:turaev-bound} we may also say that
$$2h_\pm(k) \geq \deg_u^\mp(W_k(t, A, u)).$$

\section{Knotoids with Low Height}\label{sec:applications}

\subsection{Knotoids of Height One}

\Cref{lem:main-lem} allows us to characterize knotoids of height $1$ using tangles. Let $B$ be the unit ball in $\R^3$ with labelled points $N = (0, 1, 0)$, $E = (1, 0, 0)$, $S = (0, -1, 0)$, and $W = (-1, 0, 0)$. Suppose we are given a $2$-tangle $R$ in $B$ with a strand connecting $N$ to $E$, and a strand connecting $W$ to $S$. Then we may form a knotoid $\overline R$ as in \Cref{fig:height-one-fig}. This knotoid has $h_+(\overline R) \leq 1$ and $h_-(\overline R) = 0$. Let $T$ be the set of (isotopy classes of) such tangles $R$ such that (a) no ball inside $B$ contains a nontrivial $1$-tangle and (b) $R$ is not the trivial tangle formed by two straight line segments.

\begin{figure}
\includegraphics[scale=0.3]{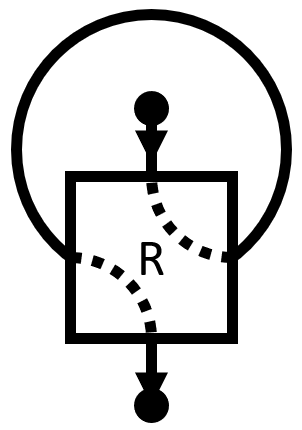}
\caption{Forming a knotoid diagram from a $2$-tangle diagram $R$.}
\label{fig:height-one-fig}
\end{figure}

A knotoid of height $1$ is prime if and only if it has no knot-type factor. Height-$1$ knotoids each have height pair $(1, 0)$ or $(0, 1)$, and the two types are in bijective correspondence via rotation. Let $U$ be the set of prime knotoids with height pair $(1, 0)$.

\begin{thm}
The map $R \mapsto \overline R$ is a bijection $T \to U$.
\end{thm}

\begin{proof}
For any $R \in T$, there is a spanning disk $D_0$ for $\tau(\overline R)$ such that $\Seq(\tau(\overline R), D_0) = (+)$ and such that $R$ can be recovered by deleting a regular neighborhood of $D_0$ and using the appropriate identification $S^3 - \nu(D_0) \cong B$. Suppose $D_1$ is another spanning disk for $\tau(\overline R)$ such that $D_0$ and $D_1$ are compatible and $D_1$ also has sign sequence $(+)$. Then on one side of the sphere $D_0 \cup D_1$ is a $1$-tangle that, by condition (a) of the definition of $T$, is unknotted. Therefore, $(\tau(\overline R), D_1)$ is isotopic to $(\tau(\overline R), D_0)$.

Suppose $\tau(\overline R)$ has height $0$ (and therefore is not in $U$). Then by \Cref{lem:main-lem} and the previous paragraph, there is a $0$-height spanning disk $D$ compatible with $D_0$. Then $D_0 \cup D$ splits $\tau(\overline R)$ into two $1$-tangles, which must both be trivial, contradicting condition (b) of the definition of $T$. Therefore, $\tau(\overline R)$ has height $1$. For all $R \in T$, no ball intersecting $\tau(\overline R)$ may contain a nontrivial $1$-tangle, and so $\overline R$ is in $U$.

Knowing that $\tau(\overline R)$ has height $1$, \Cref{lem:main-lem} now implies that no other $R'$ has $\overline{R'} = \overline R$, so the map is injective.

For any knotoid $k \in U$, we may obtain a $R \in T$ with $\overline R = k$ by finding a spanning disk $D$ for $\tau(k)$ with height $1$ and deleting a regular neighborhood of $D$. Since $k$ is prime and has height $1$, $R$ satisfies conditions (a) and (b).
\end{proof}

\subsection{Knotoids of Height Two}

Consider the following two examples of knotoids with height $2$.

\begin{exmp}
The Kinoshita knotoid $\omega$, shown in \Cref{fig:kinoshita}, is notable for being a nontrivial knotoid with trivial overpass and underpass closures. The diagram shown has a shortcut with sign sequence $(+, -)$, and $\omega$ satisfies $F_\omega(t) = t^{-1} - 2 + t$. Therefore, $h_+(\omega) = h_-(\omega) = 1$, and by \Cref{thm:unique-sequences}, $(+, -)$ is the only minimal attainable sign sequence for $\omega$.

The Kinoshita knotoid satisfies $\rev(\omega) = \rot(\omega)$. Note that neither the index polynomial nor the Turaev polynomial distinguishes $\rot(\omega)$ from $\omega$. However, by \Cref{prop:involutions-sequences}, the only minimal attainable sequence for $\rot(\omega)$ is $(-, +)$, so $\omega$ is not rotatable.

\begin{figure}
\includegraphics[scale=0.3]{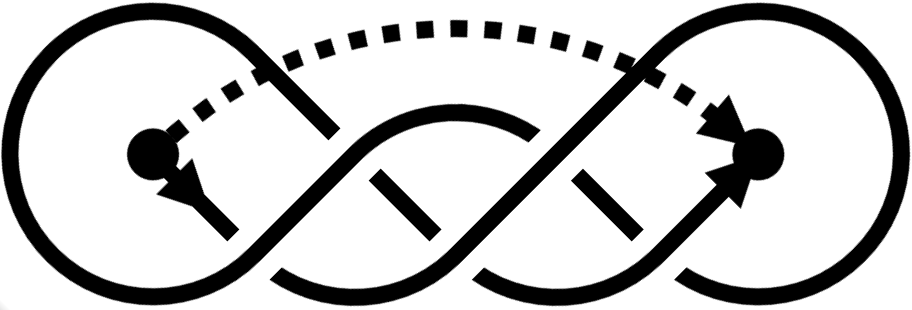}
\caption{The Kinoshita knotoid.}
\label{fig:kinoshita}
\end{figure}
\end{exmp}

\begin{exmp}
Let $k$ be the knotoid shown in \Cref{fig:cloud-fig}. The periodic diagram shown has shortcuts realizing $(+, -)$ and $(-, +)$ as attainable sign sequences. The index polynomial is $1 - t$, showing that $h_+(k) = 1$. A lower bound of $1$ for $h_-(k)$ is provided by the Turaev polynomial: the $u^2$ coefficient is $-A^{-2} + 2A^{-6} - A^{-10}$. Therefore, $(+, -)$ and $(-, +)$ are both minimal.

The information above gives us an easy way of showing that $k$ is prime: Since $k_-$ is trivial, $k$ has no knot-type factor, so to be composite it would have to be a product of two proper knotoids. One would have to have height pair $(1, 0)$, and the other $(0, 1)$, but then by \Cref{thm:seq-concatenation-d2}, only one of $(+, -)$ or $(-, +)$ would be attainable for $k$.

\begin{figure}
\includegraphics[scale=0.25]{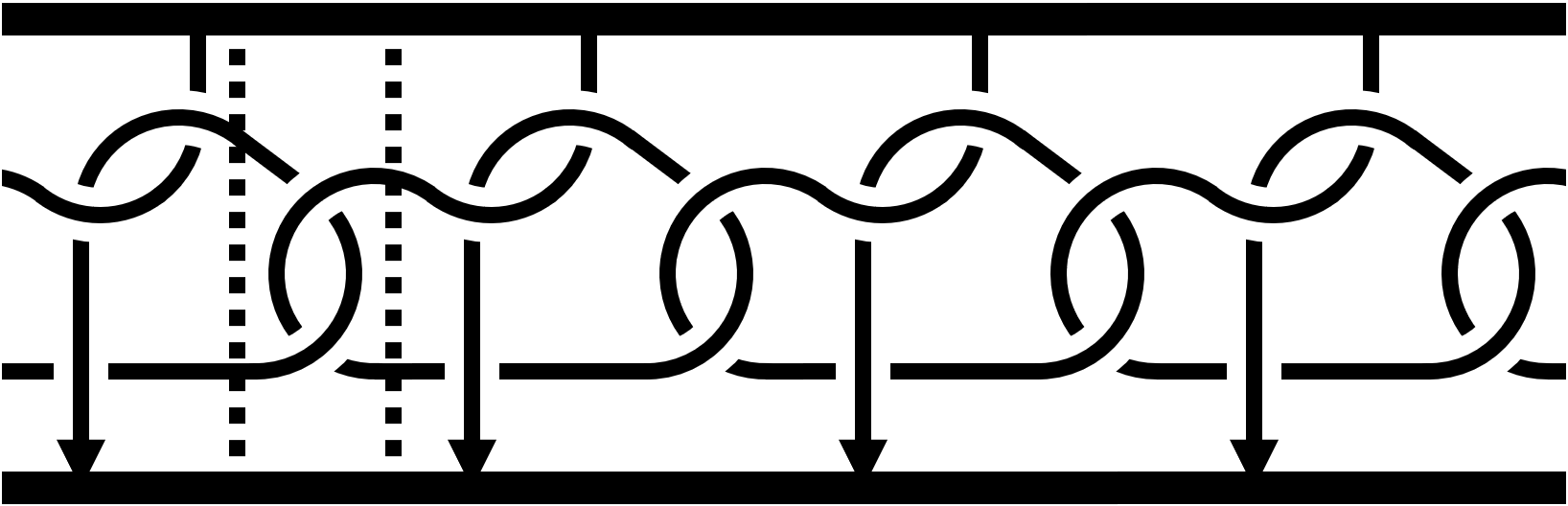}
\caption{A knotoid of height pair $(1, 1)$ with $(+, -)$ and $(-, +)$ attainable. This knotoid is reversible.}
\label{fig:cloud-fig}
\end{figure}
\end{exmp}

In general, a knotoid of height $2$ falls into one of five categories based on whether its set of minimal attainable sequences is $\{(+, +)\}$, $\{(-, -)\}$, $\{(+, -)\}$, $\{(-, +)\}$, or $\{(+, -), (-, +)\}$.
We will further divide the last category into two subcategories.

By \Cref{lem:main-lem}, if $k$ is a knotoid of height $2$ and both $(+, -)$ and $(-, +)$ are attainable sequences, then there are compatible spanning disks $D_1$ and $D_2$ respectively realizing those two sequences as attainable for $\tau(k)$.

\begin{thm}\label{thm:height-two}
Suppose $k$ is a knotoid as above. Then exactly one of the following is true.
\begin{enumerate}[(a)]
\item The disks $D_1$ and $D_2$ can be chosen in such a way that both of the intersections of $e_0$ with $D_1$ come before the intersections with $D_2$.
\item The disks $D_1$ and $D_2$ can be chosen in such a way that both of the intersections of $e_0$ with $D_2$ come before the intersections with $D_1$.
\end{enumerate}
\end{thm}

To prove \Cref{thm:height-two}, we will use a particular notion of splitting for $2$-tangles, analogous to splitting of links: Suppose $B$ be a ball with four labelled points $NE$, $SE$, $NW$, and $SW$ on $\bdy B$, and $C$ is a fixed choice of isotopy class of circles on $\bdy B - \{NE, SE, NW, SW\}$ separating $NE$ and $SE$ on one side from $NW$ and $SW$ on the other. A circle in $C$ will be called a \emph{splitting circle}. Then a $2$-tangle in $B$ will be called \emph{split} with respect to $C$ if there is a properly embedded disk (a \emph{splitting disk}), disjoint from the strands of tangle, whose boundary is a splitting circle.

\begin{lem}\label{lem:tangle-splitting}
Suppose $T$ is a $2$-tangle formed from two other tangles $R$ and $S$ in the way shown in \Cref{fig:tangle-addition}. Take $C_R$ and $C_T$ to be the classes of $\bdy(B_R \cap B_S)$ on $B_R$ and $B_T$, respectively. Then $T$ is split if and only if $R$ is split.
\end{lem}

\begin{figure}
\includegraphics[scale=0.4]{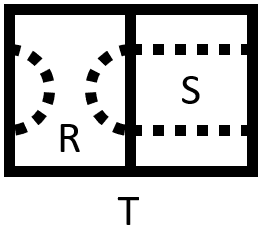}
\caption{A $2$-tangle $T$ formed as a sum of two other $2$-tangles with particular patterns. The ambient ball $B_T$ for $T$ is the union $B_R \cup B_S$ of the ambient balls for $R$ and $S$. They are attached in such a way that $B_R \cap B_S$ is a disk.}\label{fig:tangle-addition}
\end{figure}

\begin{proof}
Of course, if $R$ is split then $T$ is split. Conversely, suppose we have a splitting disk $D \subset B_T$ for $T$. Choose $D$ in such a way that $D$ is in general position and $\bdy D$ is disjoint from $B_R$. The intersection $D \cap B_R$ must separate the strands of $R$ from each other. If any components of $D \cap B_R \cap B_S$ bound disks in the punctured surface $B_R \cap B_S$, those components can be removed by cutting and capping. Also, no component of $D \cap B_R \cap B_S$ may separate the $NE$ point of $R$ from the $SE$ point of $R$ in $B_R \cap B_S$, because the circle must be nullhomotopic in $B_T - T$. Then there must be an odd number of remaining components of $D \cap B_R \cap B_S$, and they must all be separating circles for $R$. An innermost such circle in $D$ would bound a splitting disk for $R$.
\end{proof}

For a simple theta-curve $\theta$ and spanning disk $D$ with $\Seq(\theta, D) = (+, -)$, deleting a regular neighborhood of $D$ creates a $3$-tangle of the pattern shown in \Cref{fig:plus-minus-tangle}. This $3$-tangle is well-defined up to simultaneous braiding on the left and right, and the strands can be labelled as the first, second, and third strands based on the order they appear on $e_0$. Let $P_D$ be the $2$-tangle formed by deleting the first strand, and $Q_D$ the $2$-tangle formed by deleting the third. We will call $D$ \emph{(a)-split} if $P_D$ is nonsplit and $Q_D$ is split, where splitting is indicated by the jagged line. Conversely, $D$ will be called \emph{(b)-split} if $P_D$ is split and $Q_D$ is nonsplit. For spanning disks $D$ with sequence $(-, +)$, we can form $P_D$ and $Q_D$ in a similar way, but we use the opposite convention for (a)- and (b)-splitting: $D$ is (a)-split if $P_D$ is split and $Q_D$ is nonsplit.

Note that (a)- and (b)-splitting only apply to pairs $(\theta, D)$ with $h_\pm(\theta, D) = 1$, and no disk may be both (a)-split and (b)-split.

\begin{figure}
\includegraphics[scale=0.4]{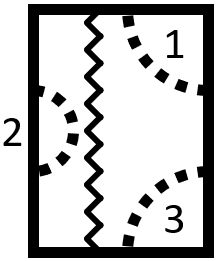}
\caption{A $3$-tangle corresponding to a spanning disk with sequence $(+, -)$. The splitting condition is determined by the jagged line.}\label{fig:plus-minus-tangle}
\end{figure}

\begin{lem}\label{lem:a-b-splitting}
If $k$ is a knotoid such that condition (a) from \Cref{thm:height-two} is true, then every minimal spanning disk for $\tau(k)$ is (a)-split. If instead (b) is true, every minimal spanning disk is (b)-split.
\end{lem}

\begin{proof}
Suppose that (a) is true of $k$. Then $k$ may be drawn as in \Cref{fig:height-two-type-a}, and the tangle $H$ must be nonsplit, as otherwise $k$ would have height $0$. The $3$-tangles corresponding to $D_1$ and $D_2$ are each formed by adding one copy of $H$ with one copy of $X$ in the appropriate order. By \Cref{lem:tangle-splitting}, $D_1$ and $D_2$ are both (a)-split. Furthermore, by another application of \Cref{lem:tangle-splitting}, if $D$ and $D'$ are any two compatible spanning disks such that one is (a)-split, then the other is (a)-split as well. Then \Cref{lem:main-lem} implies that all minimal spanning disks are (a)-split.

\begin{figure}
\includegraphics[scale=0.4]{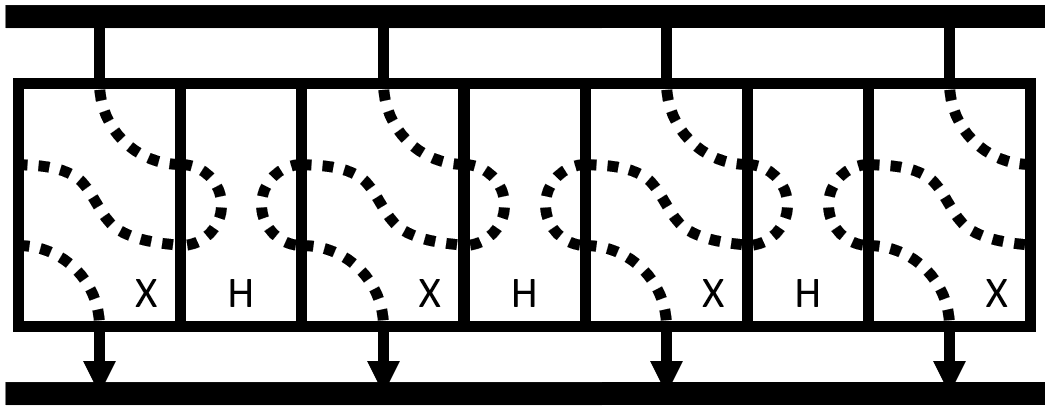}
\caption{A knotoid for which (a) is true, decomposed into tangles $H$ and $X$.}\label{fig:height-two-type-a}
\end{figure}

The same reasoning shows that if (b) is true of $k$, then all minimal spanning disks of $k$ are (b)-split.
\end{proof}

\begin{proof}[Proof of \Cref{thm:height-two}]
First we show that (a) or (b) is true. Suppose we have any choice of $D_1$ and $D_2$. By the same reasoning as in the proof of \Cref{lem:compatible-shift-moves}, a positive intersection of $e_0$ with $D_1$ must be followed by a negative $D_1$ intersection or positive $D_2$ intersection, and a negative $D_2$ intersection must be followed by a positive $D_2$ intersection or negative $D_1$ intersection. Therefore, the overall sequence of intersections is either $(+D_1, -D_1, -D_2, +D_2)$ or $(-D_2, +D_2, +D_1, -D_1)$.

That (a) and (b) cannot both be true follows from \Cref{lem:a-b-splitting}.
\end{proof}

\begin{exmp}
Consider the knotoid $k$ shown in \Cref{fig:borromean}. The index polynomial is $0$, but the Turaev polynomial tells us that the positive and negative heights are both $1$. The spanning disk corresponding to the marked shortcut is neither (a)-split nor (b)-split, because the corresponding tangles $P$ and $Q$ are both split. Therefore, $k$ is neither type (a) nor type (b), so $(+, -)$ is its only minimal attainable sequence.

\begin{figure}
\includegraphics[scale=.45]{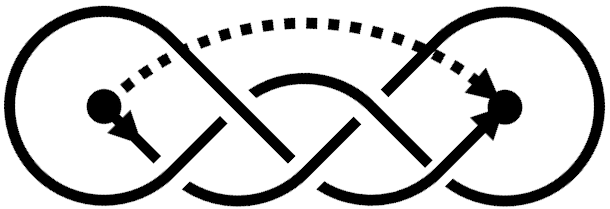}
\caption{A knotoid for which $(+, -)$ is the only minimal attainable sign sequence. It is unlike the Kinoshita knotoid in that its index polynomial is trivial.}\label{fig:borromean}
\end{figure}
\end{exmp}

We now have a partition of the set of height-$2$ knotoids into six categories: Type (a), type (b), and four categories for knotoids that each have only one minimal attainable sequence. For any knotoid $k$ with height $2$, the rotation $\rot(k)$ is in a different category from $k$. Together with \Cref{cor:odd-rotatable}, this implies the following corollary.

\begin{cor}
No proper knotoid with height below $4$ is rotatable.
\end{cor}

The author does not know if any proper rotatable knotoids exist. In \cite{BBHL18} it is shown that a knotoid cannot be rotatable if its double branched cover (see \Cref{sec:lifting}) is hyperbolic.

\nocite{*}
\printbibliography

\end{document}